\documentclass[a4page]{amsart}

\usepackage[english]{babel}
\usepackage[utf8]{inputenc}
\usepackage{amsmath}
\usepackage{url}
\usepackage{parskip}
\usepackage{amssymb}
\usepackage{amsthm}
\usepackage{mathtools,xparse}
\usepackage{graphicx}
\usepackage[colorinlistoftodos]{todonotes}
\usepackage{tikz-cd}
\usepackage{hyperref}

\DeclarePairedDelimiter{\abs}{\lvert}{\rvert}
\DeclarePairedDelimiter{\norm}{\lVert}{\rVert}
\NewDocumentCommand{\normL}{ s O{} m }{%
  \IfBooleanTF{#1}{\norm*{#3}}{\norm[#2]{#3}}_{L_2(\Omega)}%
}

\title{A Generalization of Renault's Theorem for Cartan Subalgebras}
\author{Ali I. Raad}
\address{Ali I. Raad, Department of Mathematics, KU Leuven, 200B Celestijnenlaan, 3001 Leuven, Belgium}
\email{ali.imadraad@kuleuven.be}

\begin{document}

\maketitle

\newtheorem{thm}{Theorem}[section]
\newtheorem{defn}[thm]{Definition}
\newtheorem{lem}[thm]{Lemma}
\newtheorem{cor}[thm]{Corollary}
\newtheorem{propn}[thm]{Proposition}
\newtheorem{axm}[thm]{Axiom}
\newtheorem{metthm}[thm]{Metatheorem}
\newtheorem{rmk}[thm]{Remark}
\let\olddefinition\rmk
\renewcommand{\rmk}{\olddefinition\normalfont}
\newtheorem{exmp}[thm]{Example}

\begin{abstract}
We prove a generalized version of Renault's theorem for Cartan subalgebras. We show that the original assumptions of second countability and separability are not needed. This weakens the assumption of topological principality of the underlying groupoid to effectiveness. 
\end{abstract}

\begin{section}{Introduction and statement of Results}
Jean Renault, in \cite{Ren}, characterises Cartan subalgebras of separable $C^{*}$-algebras via certain étale twisted groupoids. Specifically, we have:

\begin{thm}[Renault's Theorem, 5.2 and 5.9 in \cite{Ren}]\label{Renaults classification theorem}
Let $(G,\Sigma)$ be a twisted étale Hausdorff locally compact second countable topologically principal groupoid. Then $C_{0}(G^{0})$ is a Cartan subalgebra of $C^{*}_{r}(G,\Sigma)$. 

Conversely, if $B$ is a Cartan subalgebra of a separable $C^{*}$-algebra $A$, then there exists a twisted étale Hausdorff locally compact second countable topologically principal groupoid $(G,\Sigma)$ and an isomorphism which carries $A$ onto $C^{*}_{r}(G,\Sigma)$ and $B$ onto $C_{0}(G^{0})$. 
\end{thm}

In this paper we generalize this theorem. Specifically, for the first statement of Theorem \ref{Renaults classification theorem}, we remove the second countability assumption and also weaken topological principality to merely assuming effectiveness of the groupoid. For the converse statement we remove separability, but pay the price by obtaining a groupoid that is not necessarily second countable and thus not necessarily topologically principal. We obtain:

\begin{thm}\label{my theorem}
Let $(G,\Sigma)$ be a twisted étale Hausdorff locally compact effective groupoid. Then $C_{0}(G^{0})$ is a Cartan subalgebra of $C^{*}_{r}(G,\Sigma)$. 

Conversely, if $B$ is a Cartan subalgebra of a $C^{*}$-algebra $A$, then there exists a twisted étale Hausdorff locally compact effective groupoid $(G,\Sigma)$ and an isomorphism which carries $A$ onto $C^{*}_{r}(G,\Sigma)$ and $B$ onto $C_{0}(G^{0})$. 
\end{thm}

The following definition of a Cartan subalgebra is due to Renault in \cite{Ren}:

\begin{defn}\label{cartandef}
A $C^{*}$-subalgebra $B$ of a $C^{*}$-algebra $A$ is a Cartan subalgebra if 
\begin{itemize}
\item $B$ contains an approximate unit for $A$,
\item $B$ is a masa (maximal Abelian subalgebra) in $A$,
\item $B$ is regular in $A$ (in other words the normalizer set of $B$, $N(B) = \{a \in A : aBa^{*} \subset B \; \; \textrm{and} \; \; a^{*}Ba \subset B\}$, generates $A$ as a $C^{*}$-algebra), and finally,
\item there exists a faithful conditional expectation $P:A \twoheadrightarrow B $. 
\end{itemize}
\end{defn}

\begin{rmk}\label{review remark}
There exists a recent result by Pitts \cite{pitts} showing that the approximate unit condition in Definition \ref{cartandef} is redundant as it follows from the other conditions. 
\end{rmk}

Section 2 of this paper is devoted to summarising Renault's proofs from \cite{Ren}. Here we will highlight the main ideas in the construction, especially those ideas that we will eventually generalize.

Section 3 provides the argument for Theorem \ref{my theorem}. Here we develop the techniques required to generalize the relevant sections of Renault's proof, and highlight where they are used. 

Our approach includes using an Urysohn type lemma for locally compact Hausdorff spaces that are not necessarily second countable, in order to obtain certain separation results which are used throughout Renault's proofs in \cite{Ren}. Of course, if one assumes second countability, then paracompactness follows and hence so does normality, which yields the more standard version of Urysohn's lemma, and obtaining certain separation functions becomes trivial in this setting. We modify slightly some of Renault's proofs in \cite{Ren} in order to not assume second countability. By removing the assumptions of second countability of the groupoid and separability of the $C^*$-algebra one obtains groupoids that are not necessarily topologically principal, but effective. 

One of the advantages of Theorem \ref{my theorem} is that it may be applied to Cartan subalgebras of $C^*$-algebras that are not necessarily separable. An important class of non-separable $C^*$-algebras are the uniform Roe algebras, which are of interest as they build a link to coarse geometry (see Section 1 in \cite{wil}). These have Cartan subalgebras (see Section 6 in \cite{liren}) which fall outside that which Renault's theorem can capture. In addition, the authors of \cite{liren} obtain a distinguished Cartan subalgebra by using a slight modification of Renault's theorem, where second countability of the groupoid is weakened to $\sigma$-compactness. Of course with the more general Theorem \ref{my theorem}, this is not necessary. 

We would like to point out Corollary 7.6 in \cite{kwas} where the same conclusion as Theorem \ref{my theorem} is obtained, but via a different approach. Our approach aims at following and directly generalizing the steps in \cite{Ren}.

\textbf{Acknowledgements:} This research was conducted in 2018 as part of my PhD project in Queen Mary University of London and the University of Glasgow. This project has received funding from the European Research Council (ERC) under the European Union's Horizon 2020 research and innovation programme (grant agreement No. 817597). The author was supported by the Internal KU Leuven BOF project C14/19/088 and project G085020N funded by the Research Foundation Flanders (FWO). I would like to thank Xin Li for his supervision in this project.

\end{section}

\begin{section}{Summary of Renault's Proof}
This section serves as a summary of the constructions in \cite{Ren}, which yield Theorem \ref{Renaults classification theorem}. We assume throughout that the reader is familiar with the basic notions in the theory of étale groupoids. Information on this may be found, amongst other sources, in Chapter 3 of \cite{Put}, Chapters 1.1 and 1.2 of \cite{Ren book}, and/or Chapters 2 and 3 of \cite{Sim}. 

The first statement in Theorem \ref{Renaults classification theorem} is that $C_0(G^0)$ is a Cartan subalgebra of $C^*_r(G,\Sigma)$ for a twisted étale Hausdorff locally compact second countable topologically principal groupoid. In order to explain how this arises, we start by defining the notions of topological principality, the twist $\Sigma$, and how one obtains a $C^*$-algebra $C^*_r(G,\Sigma)$ from such groupoids. Thereafter we show that $C_0(G^0)$ is a Cartan subalgebra of $C^*_r(G,\Sigma)$, which involves showing that it satisfies the requirements of Definition \ref{cartandef}.

\begin{defn}
An étale groupoid $G$ is topologically principal if the set of points in $G^{0}$ with trivial isotropy is dense in $G^{0}$.
\end{defn}

We now summarize pages 39-41 of \cite{Ren} and pages 975-976 of \cite{Kumjian}, which define the twist and show how to obtain a $C^*$-algebra. Let $\Sigma$ be a locally compact groupoid that is also a principal $\mathbb{T}$-space. Let $G:= \Sigma / \mathbb{T}$, which is made into a topological groupoid in the natural way. 

This gives rise to a $\mathbb{T}$-bundle: $$\mathbb{T} \rightarrow \Sigma \rightarrow G.$$ We say $\Sigma$ is a \emph{twist} over $G$. In the language of exact sequences this is equivalent to a central extension: $$\mathbb{T} \times G^{0} \hookrightarrow \Sigma \twoheadrightarrow G.$$ 

It is convenient to consider another $\mathbb{T}$-bundle: $$\mathbb{T} \rightarrow \mathbb{C} \times \Sigma \rightarrow (\mathbb{C} \times \Sigma) / \mathbb{T}. $$ The $\mathbb{T}$-action is given by $t(z,\gamma) = (\overline{t}z,t\gamma)$, and of course the projection onto the base space is the canonical projection onto orbit classes $[z,\gamma]$. Set $L:= (\mathbb{C} \times \Sigma) / \mathbb{T}$ and form the complex line bundle: $$\mathbb{C} \rightarrow L \rightarrow G.$$ 

The projection map onto the base space is given by $[z,\gamma] \to \dot{\gamma}$, where $\dot{}$ is the canonical projection $\Sigma \rightarrow G$. Continuous sections of this line bundle have a representation via $\mathbb{T}$-equivariant continuous maps $\Sigma \rightarrow \mathbb{C}$ (maps satisfying $f(t\gamma)=\overline{t}f(\gamma)$ for all $t \in \mathbb{T}$).

Let $G$ be a locally compact Hausdorff groupoid with Haar system $\{\lambda_{x} : x \in G^{0}\}$, and let $\Sigma$ be a twist over $G$. We denote this pairing of a groupoid and its twist by $(G,\Sigma)$. Consider the space of compactly supported continuous sections $G \rightarrow L$, which we denote by $C_{C}(G,\Sigma)$. Define a multiplication and involution on $C_{C}(G,\Sigma)$ which turns it into a *-algebra, as follows: for $f,g \in C_{C}(G,\Sigma)$, let 

\begin{equation}\label{multiplication}
fg(\sigma) = \int\limits_{G}f(\sigma \tau^{-1})g(\tau)\mathrm{d}\lambda_{s(\sigma)}(\dot{\tau}), \; \; f^{*}(\sigma) = \overline{f(\sigma^{-1})}.
\end{equation}

For each $x \in G^{0}$, let $H_{x} = L^{2}(G_{x},L_{x},\lambda_{x})$, where $L_{x}:=p^{-1}(G_{x})$ where $p$ is the projection $L \rightarrow G$. Then define $\pi_{x}: C_{C}(G,\Sigma) \rightarrow \mathcal{B}(H_{x})$ by $$\pi_{x}(f)\zeta(\sigma) = \int\limits_{G}f(\sigma \tau^{-1})\zeta(\tau)\mathrm{d}\lambda_{x}(\dot{\tau}), \; \; \textrm{where} \; \; \zeta \in H_{x}.$$  

Finally, define the norm $\norm{f}:= \sup\limits_{x \in G^{0}}\norm{\pi_{x}(f)}_{\mathcal{B}(H_{x})}$, and complete $C_{C}(G,\Sigma)$ with respect to this norm, obtaining the reduced twisted groupoid $C^{*}$-algebra $C^{*}_{r}(G,\Sigma)$. It can be shown that for every $f \in C_{C}(G,\Sigma)$, we have $\norm{\pi_{x}(f)} \le \norm{f}_{I}:= \max \left(\sup\limits_{y \in G^{0}} \int\limits_{G} \abs{f}\mathrm{d}\lambda_{y},\sup\limits_{y \in G^{0}} \int\limits_{G} \abs{f^{*}}\mathrm{d}\lambda_{y} \right)$, which can in turn be shown to be a norm on $C_{C}(G,\Sigma)$, known as the $I$-norm. This whole construction does not assume second countability or topological principality of $G$. For thorough details, consult Chapter 2, Section 1, of \cite{Ren book}. 

Let $ supp^{\prime}(f) = \{\gamma \in G : f(\gamma) \neq 0 \}$. Renault shows in Proposition 4.1 and its consequences in \cite{Ren} that we obtain the following properties:

\begin{propn}
Let $(G,\Sigma)$ be a twisted étale locally compact Hausdorff groupoid. Then:
\begin{itemize}
\item For all $f \in C_{C}(G,\Sigma)$ we have $\abs{f(\sigma)}\le \norm{f}$ for every $\sigma \in \Sigma$, and $\int\limits_{G}\abs{f}^{2}\mathrm{d}\lambda_{x} \le \norm{f}^{2}$ for every $x \in G^{0}$.
\item The elements of $C^{*}_{r}(G,\Sigma)$ can be represented as continuous sections of the line bundle $L$. 
\item The multiplication and involution equations in (\ref{multiplication}) are valid for elements of $C^{*}_{r}(G,\Sigma)$. 
\item We have the identification $C_{0}(G^{0})= \{f \in C^{*}_{r}(G,\Sigma) : supp^{\prime}(f) \subset G^{0}\}.$
\end{itemize}
\end{propn}

Renault proves that starting with a Hausdorff étale locally compact second countable topologically principal twisted groupoid $(G,\Sigma)$, one can obtain that $C_0(G^0)$ is a Cartan subalgebra of $C^{*}_{r}(G,\Sigma)$ (we will say that $(C^{*}_{r}(G,\Sigma), C_0(G^0))$ is a Cartan pair). 

The first thing to check (as per Definition \ref{cartandef}) is that the subalgebra contains an approximate unit for the $C^*$-algebra. In the second countable case, one can construct a countable approximate unit using the $\sigma$-compactness of the unit space. In view of Remark \ref{review remark} this is of course unnecessary.

Theorem 4.2 in \cite{Ren} shows that an element of $C^*_r(G,\Sigma)$ commutes with $C_0(G^0)$ if and only if its open support is contained in the isotropy bundle $G^\prime$ which then yields that $C_0(G^0)$ is a masa in $C^*_r(G,\Sigma)$ (since $G$ is topologically principal).

Proposition 4.3 in \cite{Ren} asserts the existence of a unique faithful conditional expectation $P: C^*_r(G,\Sigma) \rightarrow C_0(G^0)$ defined by restriction. That this is a faithful conditional expectation can be checked directly by definitions, but uniqueness is obtained through topological principality and second countability of the groupoid. Indeed, Renault shows that any other conditional expectation $Q$ would have to agree with $P$ on $C_C(G,\Sigma)$, by dividing the argument into two cases. First, by considering elements in $C_C(G,\Sigma)$ whose compact support is contained in an open bisection that does not meet $G^0$, and second by considering an arbitrary element $f$ in $C_C(G,\Sigma)$ but reducing to the first case by covering the support of of $f$ by those bisections that do not meet $G^0$, and one that does, and then using a partition of unity subordinate to such a finite cover. 

In this argument, Renault makes crucial use of Urysohn's lemma which allows him to find elements in $C_C(G^0)$ that separate closed subsets from disjoint points. Of course with the space assumed second countable and locally compact, it is regular hence paracompact hence normal, and so Urysohn's lemma applies. 

Finally, one needs the regularity of $C_0(G^0)$ in $C^{*}_{r}(G,\Sigma)$. Renault shows this in Proposition 4.8 and Corollary 4.9 in \cite{Ren}. He proves that the elements of the normalizer set $N(C_0(G^0))$ are exactly those elements of $C^{*}_{r}(G,\Sigma)$ whose open support is a bisection. Since an element in $C_C(G,\Sigma)$ can be written as a finite sum of elements each of whose open support is an open bisection (this is because $G$ is étale and so has a basis of open bisections, and so one can use a partition of unity with respect to a finite cover), and each summand is a normalizer, the result follows. 

This proves how one goes from a Hausdorff étale locally compact second countable topologically principal twisted groupoid $(G,\Sigma)$ to a Cartan pair \newline $(C^*_r(G,\Sigma),C_0(G^0))$, which is the first statement in Theorem \ref{Renaults classification theorem}.

For the reverse statement, Renault starts with an arbitrary Cartan pair $(A,B)$, and constructs a Hausdorff étale locally compact second countable topologically principal twisted groupoid $(G(B),\Sigma(B))$. The construction is given in \cite{Ren}.

Start by letting $X = \mathrm{Spec}(B)$, and define $D(B) = \{(x,n,y) \in X \times N(B) \times X : \alpha_n(y) = x\}$. Here, $\alpha_n : \{x \in X : n^*n(x) > 0\} \rightarrow \{x \in X : nn^*(x) >0 \}$ is the unique homeomorphism satisfying $$n^*bn(x) = b(\alpha_n(x))n^*n(x)$$ for all $b \in B$, $x \in \{x \in X : n^*n(x) > 0\}$ (see Proposition 4.7 in \cite{Ren}). Let $\Sigma(B) = D/\sim $ where $(x,n,y) \sim (x^\prime,n^\prime,y^\prime)$ if and only if $y = y^\prime$ and there exists a $b, b^\prime \in B$ with $b(y), b^\prime(y) > 0$ and $nb=n^{\prime}b^\prime$. $\Sigma(B)$ is the given the groupoid of germs structure (see Section 3 in \cite{Ren}). $G(B)$ is defined as the image of the map taking $[x,n,y]$ in $\Sigma(B)$ to $[x,\alpha_n,y]$. The map is well-defined and $G(B)$ also inherits the groupoid of germs structure.  

We may identify the set $\mathcal{B}:= \{[x,b,x] : b \in B, b(x)\neq 0\}$ with $\mathbb{T} \times X$ via the map $[x,b,x] \to \left(\frac{b(x)}{\abs{b(x)}},x \right)$. 
Proposition 4.14 in \cite{Ren} gives an algebraic extension
$$\mathcal{B} \rightarrow \Sigma(B) \rightarrow G(B).$$

It makes use of the fact that given a point and an open neighbourhood around it in $X$, we may find a compactly supported continuous function on $X$ with compact support inside $U$.

The topology is recovered in Lemma 4.16 in \cite{Ren}, so that $\Sigma(B)$ becomes a locally trivial topological twist over $G(B)$.

Let $L(B)$ be the complex line bundle arising, as we described in this section, as $\mathbb{C} \rightarrow L(B) \rightarrow G(B)$. The aim is to construct continuous sections of this line bundle, which is equivalent to having $\mathbb{T}$-equivariant continuous maps $\Sigma(B) \rightarrow \mathbb{C}$. This is done in Lemma 5.3 in \cite{Ren}, where one defines, for $a \in A$ and $(x,n,y) \in D$, $\hat{a}(x,n,y) = \frac{P(n^*a)(y)}{\sqrt{n^*n(y)}}$ ($P$ is the conditional expectation associated to the Cartan pair), showing that this is independent of choice of representative for a class in $\Sigma(B)$ (so this map can be defined on the quotient), and is continuous and $\mathbb{T}$-equivariant. The lemma also shows that the map $\hat{}$ which sends $ a$ to $\hat{a}$ is injective and linear. Injectivity makes use of the fact that elements in $B$ can separate closed sets from points, which is automatic when $A$ is separable as it implies that $X$ is second countable and so one may apply Urysohn's lemma. 

Proposition 5.7 in \cite{Ren} argues that because the elements $\{\hat{a} : a \in A\}$ separate the points of $G(B)$, this groupoid is Hausdorff. Being a groupoid of germs it is also étale. Again this makes use of Urysohn's lemma. 

Lemma 5.8 in \cite{Ren} proves that the map $\hat{}$ is a *-algebra isomorphism from $A_C$ to $C_C(G(B),\Sigma(B))$, where $A_C$ is the linear span of elements of $N(B)$ whose image under $\hat{}$ has compact support. Theorem 5.9 in \cite{Ren} then extends this to showing that the map is an isometry with respect to the $C^*$-algebra norms, sending $A$ onto $C^*_r(G(B),\Sigma(B))$ and $B$ onto $C_0(G(B)^0)$. Separability of $A$ implies second countability of the groupoid. Topological principality is concluded from Proposition 3.6 in \cite{Ren}.

Hence the construction completes the second statement in Theorem \ref{Renaults classification theorem}. Starting from a Cartan pair $(A,B)$ with $A$ separable we obtain a locally compact Hausdorff étale topologically principal second countable twisted groupoid $(G(B),\Sigma(B))$.

The two procedures, going from twisted groupoids to Cartan pairs and from Cartan pairs to twisted groupoids are inverse to each other. Indeed, Proposition 4.15 in \cite{Ren} tells us that if $(G,\Sigma)$ is a locally compact Hausdorff étale topologically principal second countable twisted groupoid and we let $A=C^*_r(G,\Sigma)$ and $B=C_0(G^0)$ then we obtain an isomorphism of extensions: 

\textbf{Diagram A}
\begin{center}
\begin{tikzcd}
\mathcal{B} \arrow[r] \arrow[d] & \Sigma(B) \arrow[r] \arrow[d] & G(B) \arrow[d] \\
\mathbb{T} \times G^{0} \arrow[r] & \Sigma \arrow[r] & G
\end{tikzcd}
\end{center}

We already have noted the isomorphism of the left vertical arrow. The isomorphism of the right vertical arrow is given by Proposition 4.13 in \cite{Ren}, and the middle vertical arrow is discussed in Proposition 4.15 in \cite{Ren}. In Section 3 we will give alternate proofs of these statements that do not rely on any second countability or topological principality assumptions. 

Diagram A tells us that the process $$(G,\Sigma) \dashrightarrow (A=C^*_r(G,\Sigma),B=C_0(G^0)) \dashrightarrow (G(B),\Sigma(B))$$ is the identity (up to isomorphism). Theorem 5.9 in \cite{Ren} tells us that for a Cartan pair $(A,B)$, the process $$(A,B) \dashrightarrow (G(B),\Sigma(B)) \dashrightarrow (C^*_r(G(B),\Sigma(B)), C_0(G(B)^0))$$ is the identity (up to isomorphism).

Note that the automorphism group of the twisted groupoids can thus be identified with the automorphism group of the associated Cartan pair. 
\end{section}

\begin{section}{Generalizing Renault's Theorem}
This section proves Theorem \ref{my theorem}. We will first prove the first statement in the theorem by focusing on the areas in Renault's construction that make use of the second countability of the groupoid or the topological principality.

We will need the following definition: 

\begin{defn}
An étale groupoid $G$ is called \emph{effective} if $\mathrm{int}(G^\prime)=G^0$ (where $G^\prime$ denotes the isotropy bundle).
\end{defn}

We will require certain separation properties for locally compact Hausdorff spaces, which are used implicitly throughout \cite{Ren}, and which are standard when the topological space is second countable. However, we verify that we have the required results even for non-second countable spaces. 

\begin{lem}\label{separation}
Let $X$ be a locally compact Hausdorff space. Then 
\begin{enumerate}
\item Given a compact subset $K$ of $X$ and an open set $U$ of $X$ such that $K \subset U \subset X$, there exists $b \in C_{0}(X)$ with $b\equiv 1$ on $K$, and 0 outside $U$.
\item Given a closed subset $C \subset X$ and a point $x \in X$ disjoint from $C$, there is $b \in C_{0}(X)$ where $b(x) = 1$ and $b\vert_ C \equiv 0$.
\item Given an open set $U \subset X$ containing a point $x$ there exists an open set $V$ containing $x$ such that $\overline{V}$ is a compact subset of $U$.
\end{enumerate}
\end{lem}

\begin{proof}
Claim (1) is Urysohn's lemma for locally compact Hausdorff spaces (see \cite{Rud}, 2.12 or Proposition 15 in Chapter 9, Section 4 of \cite{royd}). Since locally compact Hausdorff spaces are regular, we may use (1) to directly get (2). 

Let us prove (3). By local compactness, there is an open set $O$ containing $x$ and a compact set $K$ containing $O$. Let $D=K \cap U^C$, which is a closed subset of a compact set, hence compact. By regularity, we may find disjoint open sets $O_D$ containing $D$ and $O_x$ containing $x$. Let $V=O_x \cap O$ which is open and contains $x$, and note that $V \subset \overline{V} \subset  \overline{O} \subset K$, hence $\overline{V}$ is compact. Now note also that $\overline{V} \subset \overline{O_x} \subset O_D^C \subset D^C = K^C \cup U$. Since $\overline{V} \subset K$ it follows $\overline{V} \subset U$.
\end{proof}

The following proposition is certainly well-known but we do not have a reference for it.

\begin{propn}
Let $(G,\Sigma)$ be a twisted étale locally compact Hausdorff groupoid. Then $C_{0}(G^{0})$ contains an approximate unit for $C^{*}_{r}(G,\Sigma)$.
\end{propn}

\begin{proof}
Let $(e_i)_{i \in I}$ be an approximate unit for $C_{0}(G^{0})$. Then it converges uniformly to 1 on any compact subset of $G^0$. Indeed, for any compact $L \subset G^{0}$, there exists, by Lemma \ref{separation}, $b \in C_0(G^0)$ such that $b \equiv 1$ on $L.$ Since $\norm{e_ib-b}_{\infty} \to 0$ we have that the approximate unit converges to 1 uniformly on $L$. 

Now let $f \in C_C(G,\Sigma)$ with compact support $K$. It is clear that the net $(e_i*f)_{i \in I} = ((e_i\circ r)f)_{i \in I}$ converges uniformly to $f$ on $K$, as the net $(e_i)_{i \in I}$ converges uniformly to 1 on $r(K)$. Hence $(e_i*f)_{i \in I}$ converges to $f$ in the inductive limit topology, which by Proposition II.1.4 (i) in \cite{Ren book} implies that it converges in the $I$-norm and hence in $C^{*}_{r}(G,\Sigma)$. The same can be said for $(f*e_i)_{i \in I}$. As by definition the net $(e_i)$ is bounded, the above holds when replacing $f$ by any $a \in C^{*}_{r}(G,\Sigma)$.
\end{proof}

In order to get that $C_0(G^0)$ is a masa in $C^*_r(G,\Sigma)$, it suffices to assume that $G$ is effective in Theorem 4.2 in \cite{Ren} rather than topologically principal.

As was stated in Section 2, in order to get a unique faithful conditional expectation in Proposition 4.3 in \cite{Ren}, Renault makes use of Urysohn's lemma. We may now use Lemma \ref{separation} (2) instead and obtain the same result. Additionally, assuming effectiveness rather than topological principality suffices.

To show regularity, note that for an étale and locally compact groupoid $G$ we have that every $f \in C_C(G,\Sigma)$ is the linear span of continuous sections compactly supported by open bisections, and such elements are normalizer elements (by Proposition 4.8 (i) in \cite{Ren}). 

Together, this gives the first statement of Theorem \ref{my theorem}. For the reverse statement, recall from Section 2 that Renault proves that there is an extension $$\mathcal{B} \rightarrow \Sigma(B) \rightarrow G(B)$$ by making use of the fact that one can find a compactly supported continuous function with support inside a given open set. We may now instead use Lemma \ref{separation} (3) for this. 

The rest of the arguments all use separation properties that are found in Lemma \ref{separation} without having to allude to second countability. Of course in Theorem 5.9 in \cite{Ren} one would not get a second countable groupoid $G(B)$ if separability of $A$ is removed, and hence $G(B)$ might not be topologically principal. However, being an étale groupoid of germs, it is automatically effective. Thus the second statement of Theorem \ref{my theorem} is obtained.

In order to get Diagram A without second countability or topological principality assumptions, Proposition 4.13 of \cite{Ren} must be slightly modified. Renault proves that given an open bisection $S$, there exists $n \in N(B)$ for which $S=supp^\prime(n)$. In order to achieve this, Renault claims the existence of an element in $C_0(G^0)$ whose open support is exactly $s(S)$. Of course when the space is second countable and locally compact, it is $\sigma$-compact and it is easy to obtain such an element. 

Given that we are not assuming second countability, we can only work with a slightly weaker version of this argument. To obtain Proposition 4.13 in \cite{Ren} without second countability it suffices to localise to an open set contained in $S$, as the groupoid of germs induced by $\alpha(\mathcal{S}),$ where $\mathcal{S}$ is the set of all open bisections, is the same as that induced by $\alpha(\mathcal{S}^\prime)$, where $\mathcal{S}^\prime$ is a refinement of $\mathcal{S}$ (for the definition of $\alpha$ see Section 3 in \cite{Ren}). Theorem 12 in the Appendix of \cite{Fell} ensures that we have a non-vanishing continuous section for the associated line bundle $L$ on a neighbourhood $T$ of $g \in S$, contained in $S$. Hence $L\vert_T$ is trivializable. We may use Lemma \ref{separation} to say that there exists a $c \in C_0(G^0)$ with compact support inside $s(T)$. Hence $U=supp^\prime(c)$ is an open set inside $s(T)$, and by the fact that $s: T \rightarrow s(T)$ is a homeomorphism we can pull $U$ back to an open bisection $V \subset T$. Restricting attention to $V$ we have that $L\vert_V$ is trivializable.

Let $u: V \rightarrow L$ be a non-vanishing section, and without loss of generality assume $\norm{u(g)} = 1$ for all $g \in V$. Define $n: G \rightarrow L$ by $n(g) = u(g)c(s(g))$ if $g \in V$, and 0 otherwise. There exists a net $\{h_{j}\}_{j \in J}$ in $C_{C}(U)$ converging uniformly to $c$. Hence $uh_{j} \in C_C(G,\Sigma)$ converges uniformly to $n$. Hence it converges in the $I$-norm as this coincides with the supremum norm on $C_0(G^0)$, and hence in the $C^*$-algebra norm. Hence $n \in A$ with $\mathrm{supp}^\prime(n)=V$, and hence $n \in N(B).$ The remaining parts of the proof work by just assuming effectiveness of the groupoid. 

 To get the isomorphism of the middle arrow in Diagram A we offer an alternative proof to Proposition 4.15 of \cite{Ren}, thanks to Xin Li:

\begin{propn}\label{commuting diagram}
Let $A = C^{*}_{r}(G,\Sigma)$ and $B=C_{0}(G^{0})$, for a twisted étale Hausdorff locally compact effective groupoid $(G,\Sigma)$. Then we have a canonical isomorphism of extensions:
\begin{center}
\begin{tikzcd}
\mathcal{B} \arrow[r] \arrow[d] & \Sigma(B) \arrow[r] \arrow[d] & G(B) \arrow[d] \\
\mathbb{T} \times G^{0} \arrow[r] & \Sigma \arrow[r] & G
\end{tikzcd}
\end{center}

\end{propn}

\begin{proof}
We have already shown the left and right vertical arrows to be isomorphisms. We just need an isomorphism of the middle vertical arrow making the diagram commute. This is achieved as follows. We define a map $\Sigma(B) \rightarrow \Sigma$ by $[x,n,y] \to \frac{n(\sigma)}{\abs{n(\sigma)}}\sigma$, where $\sigma \in \Sigma$ chosen so that $\dot{\sigma} \in supp^{\prime}(n)$ with $s(\dot{\sigma}) = y$ and $r(\dot{\sigma}) = x$. The inverse $\Sigma \rightarrow \Sigma(B)$ is defined by sending $\sigma \to \left[x,\frac{\overline{n(\sigma)}}{\abs{n(\sigma)}}n,y \right]$ where $n \in N(B)$ chosen so that $n(\sigma) \neq 0$, and $y = s(\dot{\sigma})$, $x = r(\dot{\sigma})$. It is a tedious but straightforward task to check that these maps are well-defined groupoid homomorphisms, and are inverse to each other.   
\end{proof}

Finally, we remark that since the proof of Proposition 5.11 in \cite{Ren} does not use the separability of the $C^*$-algebra nor the second countability of the groupoid, it follows that for any Cartan pair $(A,B)$ we have that $B$ satisfies the unique extension property if and only if the groupoid $G(B)$ is principal.

\end{section}

\end{document}